\documentclass[10pt, a4paper, reqno]{amsart}

\usepackage{pdfpages}
\usepackage{amssymb,amsmath,amsfonts,amsthm,commath}
\usepackage{calc,extarrows}
\usepackage{enumitem}
\usepackage{graphicx}
\usepackage{graphics}
\usepackage{mathrsfs}
\usepackage{url}
\usepackage{xcolor}
\usepackage{emptypage}
\usepackage{bbm}
\usepackage{float}
\usepackage[utf8]{inputenc} 

\usepackage{hyperref}
\definecolor{bluepigment}{rgb}{0.2, 0.2, 0.6}
\hypersetup{
	colorlinks=true,
	linkcolor=bluepigment,
	citecolor=bluepigment,
	filecolor=black,      
	urlcolor=black,
}

\usepackage[numbers,sort&compress]{natbib}

\newcommand{\Z}{\mathbb Z}

\newcommand{\N}{\mathbb N}
\newcommand{\R}{\mathbb R}

\renewcommand{\P}{\mathbb P}

\newcommand{\E}{\mathbb E}

\newtheorem{theorem}{Theorem}[section]
\newtheorem{lemma}{Lemma}[section]

\newtheoremstyle{remark}{3pt}{3pt}{}{}{\bfseries}{:}{.5em}{}
\theoremstyle{remark}
\newtheorem{remark}{Remark}[section]

\begin{document}
	\title[Last Progeny Modified BRW]{Large Deviations for the Right-Most Position of a Last Progeny Modified Branching Random Walk}
	\author[Partha Pratim Ghosh]{Partha Pratim Ghosh}  
	\date{January 28, 2022}
	\address[Partha Pratim Ghosh]{Theoretical Statistics and Mathematics Unit \\
		Indian Statistical Institute, Delhi Centre \\ 
		7 S. J. S. Sansanwal Marg \\
		New Delhi 110016 \\
		INDIA}         
	\email{p.pratim.10.93@gmail.com} 	
	
	\maketitle
	
	\begin{abstract}		
In this work, we consider a modification of the usual \emph{Branching Random Walk (BRW)}, where we give certain independent and identically distributed  (i.i.d.) displacements to all the particles at the $n$-th generation, which may be different from the driving increment distribution. This model was first introduced by
Bandyopadhyay and Ghosh \cite{BG21} and they termed it as 
\emph{Last Progeny Modified Branching Random Walk (LPM-BRW)}. Under very minimal assumptions, we derive the
\emph{large deviation principle (LDP)} for the right-most position of a particle
in generation $n$. As a byproduct, we also complete the LDP for the
classical model, which complements the earlier work by 
 Gantert and  H\"{o}felsauer  \cite{GH18}.
	\end{abstract}
	
	\vspace{0.1in}
	
	\vspace{0.1in}



\section{Introduction}
\subsection{Background and Motivation}
A branching random walk on the real line is a discrete-time stochastic process, which can be described as follows:

Let $X$ and $N$ be two random variables taking values in $\R$ and $\N$, respectively.   At the $0$-th generation, we start with an initial particle at the origin. At time $n = 1$, the particle dies and gives birth to a random number of offspring, distributed according to $N$. The offspring are then displaced from their parent's position by i.i.d. copies of $X$. For $n\geq2$, the particles at generation $(n-1)$ behave independently and identically of the particles up to generation $(n-1)$.

If we denote the number of particles in generation $n$ by $N_n$, then from the definition, it follows that $\{N_n\}_{n\geq0}$ is a Galton-Watson branching process with progeny distribution given by $N$. So the backbone of the process is a branching process tree with weighted edges. Here the weights represent the displacements of the particles relative to their respective parent. We write $|v|=n$ if an individual $v$ is in the  $n$-th generation, and its position $S(v)$ is defined as the sum of the edge-weights of the unique path connecting $v$ to the root.  We shall call the process $\{S(v):|v|=n\}_{n\geq0}$ a \textit{Branching Random Walk (BRW)}. 

In this article, we consider a modified version of the BRW. The modification occurs only at the last generation, where we add i.i.d. displacements of a specific form. There are two parameters of this model. One is a positively supported measure, $\mu$, and the other is a positive real number, $\theta$, which should be viewed as a scaling parameter for the extra shift  we give to each particle at the $n$-th generation. The modification is as follows. At a generation $n\geq1$, we give additional displacement to each of the particles at the generation $n$, which are of the form $\frac{1}{\theta}\log(Y_v/E_v)$, where  $\{Y_v\}_{|v|=n}$ are i.i.d.  $\mu$,  $\{E_v\}_{|v|=n}$ are i.i.d.  $\mbox{Exponential}\left(1\right)$, and these two sequences are independent of each other and also of the BRW. This model was first introduced by Bandyopadhyay and Ghosh~\cite{BG21} and they refer to this new process as a  \textit{Last Progeny Modified Branching Random Walk (LPM-BRW)}. We denote by $R_n$ and $R_n^*\equiv R_n^*(\theta,\mu)$  the right-most positions of the $n$-th generation particles of the BRW and the LPM-BRW, respectively, i.e.,
\begin{equation}
	R_n:= \max_{|v|=n} S(v),\quad\quad R_n^*(\theta,\mu):=\max_{|v|=n} \Big\{ S(v)+\frac{1}{\theta}\log(Y_v/E_v) \Big\}.
	\label{def:rnrn*}
\end{equation}

The main motivation to study this model is that, due to the specific form of the additional shift at the last generation, there is a nice coupling of $R_n^*$ with a \textit{linear statistic} associated with BRW, which for $\mu=\delta_1$ becomes the well-known Biggins' martingale (see Bandyopadhyay and Ghosh~\cite{BG21}). On the other hand, as $\theta$ increases, $R_n^*$ becomes closer and closer to $R_n$. This novel connection is in fact the reason why the model intrigued us.


Throughout this paper, we assume the followings:
\begin{itemize}[leftmargin=1.65cm] 
	\item[($\mathbf{A1}$)] \hypertarget{A1}{}
	The random variable $X$ is non-degenerate, i.e., $\P( X=t )<1$ for any $t\in\R$, and its moment-generating function is finite everywhere, i.e., for all $\lambda\in\R$,
	\[m(\lambda):=\E\big[e^{\lambda X}\big]<\infty.\]

	\item[($\mathbf{A2}$)] \hypertarget{A2}{}
	The underlying branching process is non-trivial, and the extinction probability   is zero, i.e., $\P(N=1)<1$, and $\P(N=0)=0$. Also, $N$ has finite $(1+p)$-th moment for some $p>0$.
	
	\item[($\mathbf{A3}$)] \hypertarget{A3}{}
	For all $k\in\Z$,
	\[\int_{0}^{\infty}x^{k }\,d\mu(x) <\infty.\]
\end{itemize}

We denote $\phi(\lambda):=\log m(\lambda)$, and $\nu(\lambda):= \phi(\lambda)+\log \E[N]$.
Note that $\nu$ is strictly convex and infinitely differentiable under assumptions (\hyperlink{A1}{$\mathbf{A1}$}) and (\hyperlink{A2}{$\mathbf{A2}$}) (see Proposition A.2 of Bandyopadhyay and Ghosh~\cite{BG21}). 
We define
\[
\theta_0:=\inf\Big\{ \theta>0 : \frac{\nu(\theta)}{\theta}= \nu'(\theta) \Big\}.
\]
Since $\nu(\theta)$ is strictly convex, the above set is at most singleton. If it is a singleton, then  $\theta_0$ is the  unique point in $(0,\infty)$ such that a tangent from the origin to the graph of $\nu(\theta)$ touches the graph at $\theta=\theta_0$. And if it is empty, then by definition $\theta_0$ takes value $\infty$, and there is no tangent from the origin to the graph of $\nu(\theta)$ on the right half-plane.

Under fairly general assumptions on the distribution of $X$ and $N$,  Hammersley~\cite{H74}, Kingman~\cite{K75}, and Biggins~\cite{B76}   showed that 
\begin{equation}
	\frac{R_n}{n}\rightarrow \frac{\nu(\theta_0)}{\theta_0} \text{, a.s., as } n\rightarrow\infty.
	\label{conv:rn}
\end{equation}

Similar convergence result for LPM-BRW was proved by Bandyopadhyay and Ghosh \cite{BG21}. They showed that for any $\theta>0$, almost surely
\begin{equation}
	\frac{R_n^*(\theta,\mu)}{n} \rightarrow c(\theta):= \left\{ \begin{array}{ll}
		\frac{\nu(\theta)}{\theta}, & \text{ if } \theta<\theta_0\leq\infty;\\[.25cm]
		\frac{\nu(\theta_0)}{\theta_0}, & \text{ if } \theta_0\leq\theta<\infty.
	\end{array} \right. 
	\label{conv:rn*}
\end{equation}
Therefore, we have 
\begin{align*}
	&\lim_{n\rightarrow\infty}\P\Big( \frac{R_n^*(\theta,\mu)}{n}> x \Big)=0 \text{ for } x>c(\theta); \text{ and }\\[.15cm]
	& \lim_{n\rightarrow\infty}\P\Big( \frac{R_n^*(\theta,\mu)}{n}< x \Big)=0 \text{ for } x<c(\theta).
\end{align*}
This paper investigates the exponential decay rates of these probabilities, which is in essence a large deviation (LDP) problem.



\subsection{Main Results}
Let $\{X_n\}_{n\geq 1}$ be i.i.d. copies of $X$. We define $\mathsf{S}_n:=\sum_{i=1}^n X_i$.
It follows from  Cram\'{e}r's theorem (see Dembo and Zeitouni~\cite{DZ10}) that the laws of $\{\mathsf{S}_n/n\}_{n\geq1}$ satisfy the large deviation principle with the rate function 
\[
I(x):=\sup_{\lambda\in\R}\left\{ \lambda x-\phi(\lambda)\right\} \text{ for } x\in\R.
\]

	From Theorem 1 of Rockafellar~\cite{RTR67}, we know that $I(x)$ is strictly convex and differentiable on the interior of its effective domain $\mathcal{D}_I:=\{x\in \R : I(x)<\infty\}$ with
	$ I'(x)=(\phi')^{-1}(x)$.
	This implies
	$I'\left(\E[X]\right)=0$ and $\lim_{x\,\downarrow\,\inf \mathcal{D}_I } I'(x)=-\infty$.
	Therefore, whenever $\rho:=-\log\P(N=1)$ is finite, there exists a unique point $a_{\theta}^{\rho}\in \left(\inf \mathcal{D}_I ,\E[X]\right)$ such that a tangent from the point $\left(c(\theta),0\right)$ to the graph of $I(x)+\rho$ touches the graph at $x=a_{\theta}^{\rho}$, i.e., $a_{\theta}^{\rho}$ satisfies
	\begin{equation*}
		\frac{I\left(a_{\theta}^{\rho}\right)+\rho}{a_{\theta}^{\rho}-c(\theta)}=I'\left(a_{\theta}^{\rho}\right).\label{athetarho}
	\end{equation*}

 We denote
$d(\theta):= \max\left\{c(\theta), \phi'(\theta) \right\}$.
Then we have

\begin{theorem}
	\label{ratefn}
	The laws of $\{  {R_n^* (\theta,\mu) /n} \}_{n\geq1}$ satisfy the large deviation principle with the rate function 
	\[
	\Psi_{\theta}(x):=\left\{
	\begin{array}{lll}
		\theta x-\phi(\theta)-\log\E[N],&\text{if } x> d(\theta);&\text{\hyperref[proofi]{(i)}}\\[.15cm]
		 I(x)-\log\E[N],&\text{if } c(\theta)<x\leq d(\theta);&\text{\hyperref[proofii]{(ii)}}\\[.15cm]
		 0,&\text{if }x=c(\theta);&\text{\hyperref[proofiii]{(iii)}}\\[.15cm]
		 I'\big(a_{\theta}^{\rho}\big)\left(x-c(\theta)\right),
		&\text{if $a_{\theta}^{\rho}\leq x< c(\theta)$ and $\rho<\infty$};&\text{\hyperref[proofivv]{(iv)}}\\[.15cm]
		 I(x)+\rho ,&\text{if $x< a_{\theta}^{\rho}$ and $\rho<\infty$};&\text{\hyperref[proofivv]{(v)}}\\[.15cm]
		 \infty,& \text{if $x<c(\theta)$ and $\rho=\infty$}.&\text{\hyperref[proofvi]{(vi)}}
	\end{array}
	\right.
	\]
\end{theorem}

While proving our main result, we also observe that we can complete the LDP for $\{R_n/n\}_{n\geq1}$, which was proved by Gantert and  H\"{o}felsauer~\cite{GH18} but only partially.

\begin{theorem}
	\label{ratefnrn}
	The laws of $\{  {R_n /n} \}_{n\geq1}$ satisfy the large deviation principle with the rate function 
	\[
\Phi (x):=\left\{
\begin{array}{lll}
	 I(x)-\log\E[N],&\text{if } x>c(\theta_0);&\text{{(i)}}\\[.15cm]
	 0,&\text{if }x=c(\theta_0);&\text{{(ii)}}\\[.15cm]
	 I'\big(a_{\theta_0}^{\rho}\big)\left(x-c(\theta_0)\right),
	&\text{if $a_{\theta_0}^{\rho}\leq x< c(\theta_0)$ and $\rho<\infty$};&\text{{(iii)}}\\[.15cm]
	 I(x)+\rho, &\text{if $x< a_{\theta_0}^{\rho}$ and $\rho<\infty$};&\text{{(iv)}}\\[.15cm]
	 \infty,& \text{if $x<c(\theta_0)$ and $\rho=\infty$}.&\text{{(v)}}
\end{array}
\right.
\]
\end{theorem}
\begin{remark}
	The parts (i), (ii), (iii), and (iv) of Theorem~\ref{ratefnrn} were proved by Gantert and  H\"{o}felsauer~\cite{GH18}, but part (v) was unsolved in their paper. As the anonymous referee  pointed out, this part was recently proved by Chen and He~\cite{CH20}. But at the time of writing this article, the author did not know this and therefore has given an alternative proof. Also, parts (iii) and (iv) of Theorem~\ref{ratefnrn} calculated by Gantert and  H\"{o}felsauer~\cite{GH18} have been simplified here.  Notice here that the rate function in Theorem~\ref{ratefnrn} is  similar to that of the \textit{Branching Brownian  Motion (BBM)}  calculated by Derrida and Shi~\cite{DS17}. 
	As a result, we also see similarities in the figures in Section~\ref{sec:example} and those in Derrida and Shi~\cite{DS17}. 
\end{remark}

\subsection{Outline}
The article is organized as follows. In Section~\ref{sec:example}, we illustrate our main results with a few examples. We give the proofs of our main results in Section~\ref{sec:proofs}. Finally, in Section~\ref{sec:comparision}, we compare the rate function for $\{  {R_n^* /n} \}_{n\geq1}$ with that of $\{  {R_n /n} \}_{n\geq1}$.

\section{Examples}
\label{sec:example}
As an illustration, in this section we consider two specific examples.
Our first example is when $N$ takes value $2$ with probability 
$1$, $X \sim \mbox{N}(0,1)$, $\theta=3$, and $\mu=\delta_1$. Then, as displayed in Figure~\ref{fig_ex1}, the large deviation rate function for the laws of
$\{  {R_n^* (3,\delta_1) /n} \}_{n\geq1}$ is
\[
f_1(x)=\left\{ \begin{array}{ll}
	3x-\frac{9}{2}-\log 2 ,& \text{ if } x\geq 3;\\[.15cm]
	\frac{x^2}{2} -\log 2 ,& \text{ if } \sqrt{2\log 2}\leq x\leq 3;\\[.15cm]
	\infty ,& \text{ if } x<\sqrt{2\log 2}.
\end{array} \right.
\]
\begin{figure}[ht]
\begin{minipage}{.49\textwidth}
	\centering
		\includegraphics[width=.9\textwidth]{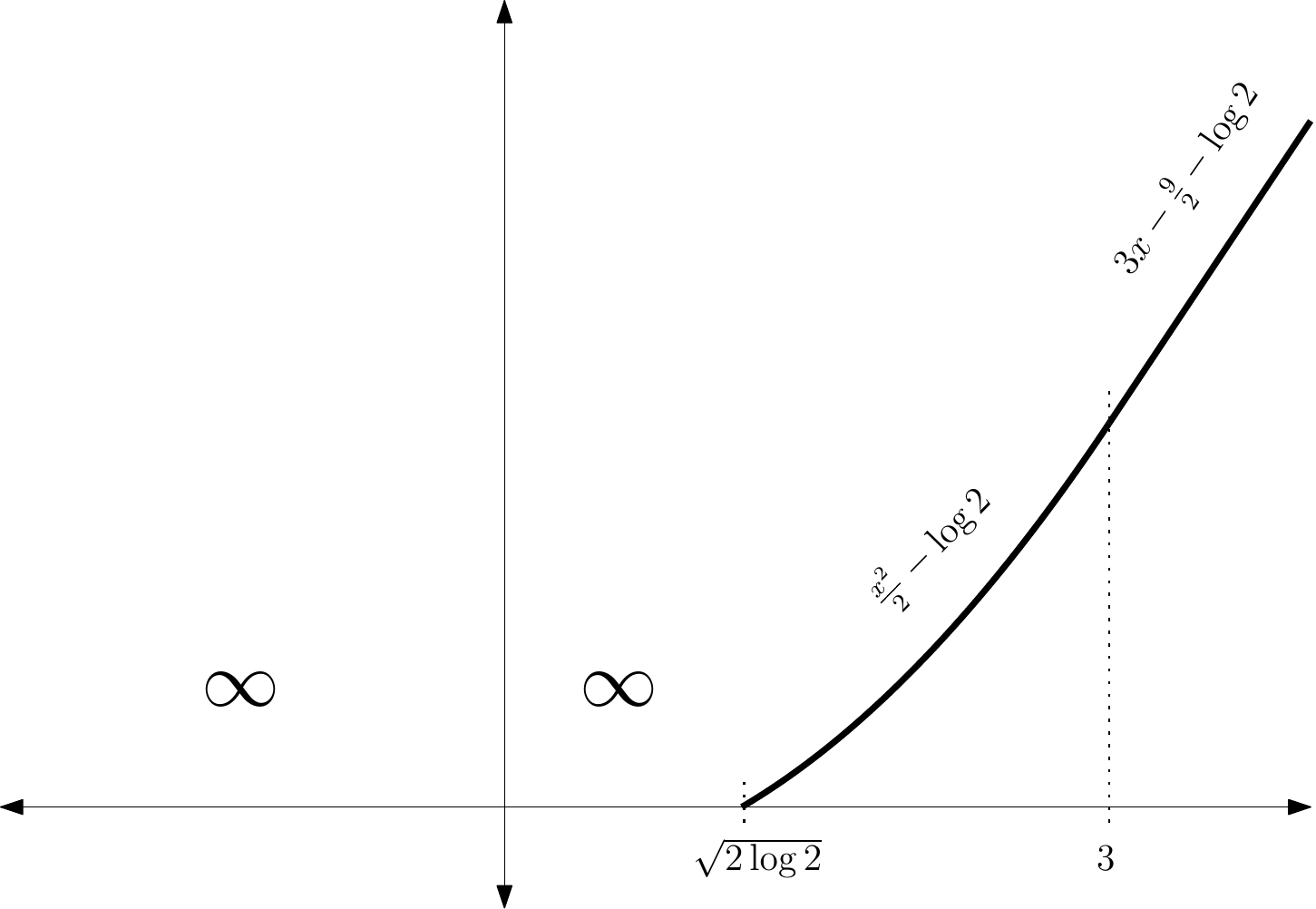}\\
		\caption{Graph of $f_1$}\label{fig_ex1}
	\end{minipage}
\begin{minipage}{.49\textwidth}
\centering
\includegraphics[width=.9\textwidth]{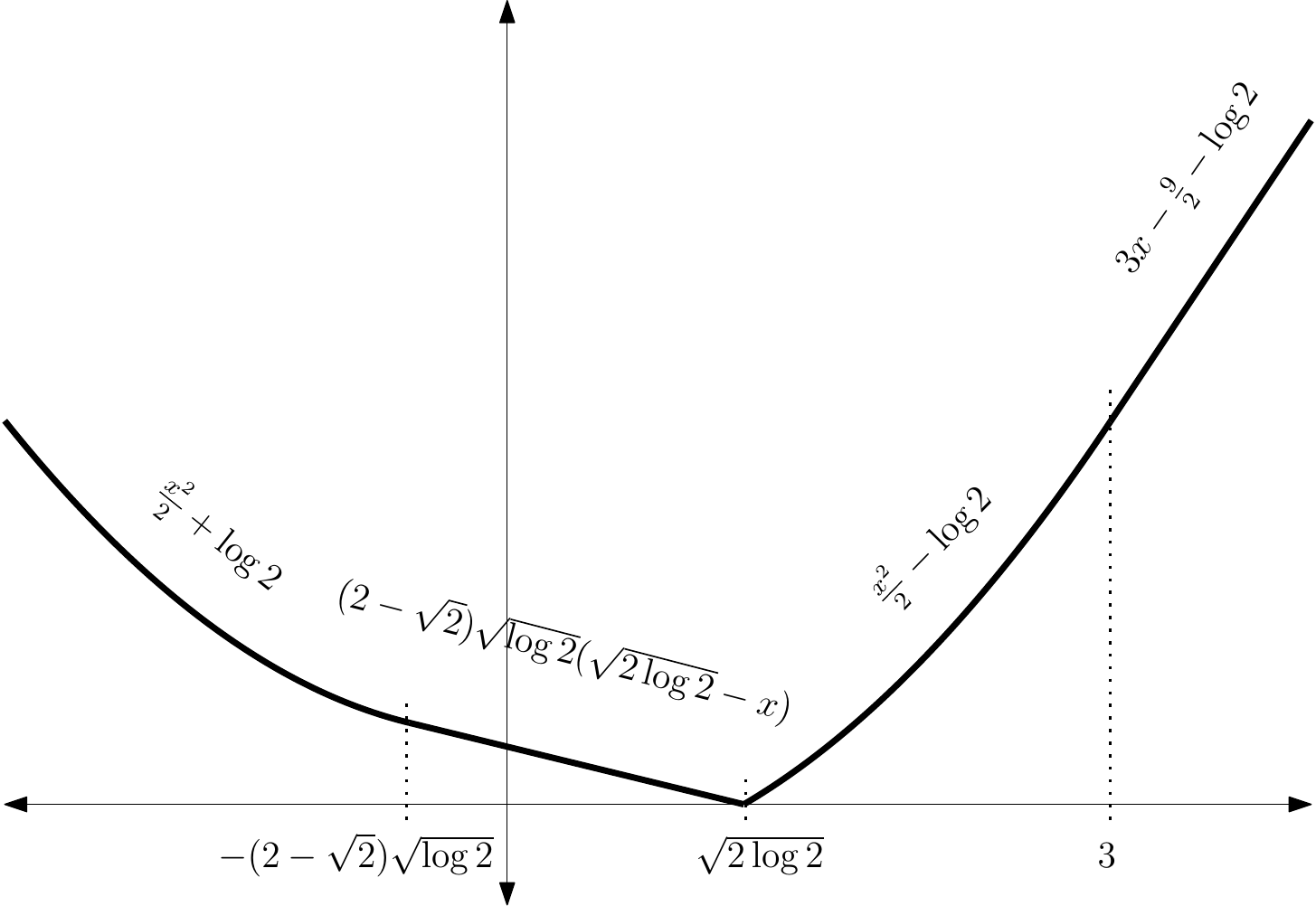}\\
\caption{Graph of $f_2$}\label{fig_ex2}
\end{minipage}
\end{figure}
On the other hand, if $N$ takes the value $1$ with probability $1/2$ and 
$3$ with probability $1/2$, and $X$, $\theta$, and $\mu$ are as in the previous example, then, as demonstrated in Figure~\ref{fig_ex2}, the large deviation rate function for the laws of
$\{  {R_n^* (3,\delta_1) /n} \}_{n\geq1}$ is
\[
f_2(x)=\left\{ \begin{array}{ll}
	3x-\frac{9}{2}-\log 2 ,& \text{ if } x\geq 3;\\[.15cm]
	\frac{x^2}{2} -\log 2 ,& \text{ if } \sqrt{2\log 2}\leq x\leq 3;\\[.15cm]
	\big(2-\sqrt{2}\big)\sqrt{\log 2}\big(\sqrt{2\log 2}-x\big) ,& \text{ if } -\big(2-\sqrt{2}\big)\sqrt{\log 2}\leq x\leq\sqrt{2\log 2};\\[.15cm]
	\frac{x^2}{2} +\log 2 ,& \text{ if } x\leq-\big(2-\sqrt{2}\big)\sqrt{\log 2}.
\end{array} \right.
\]
\section{Proofs of the Main Results}
\label{sec:proofs}

\subsection{Proof of Theorem~\ref{ratefn}}
The general strategy to prove this theorem is to give an upper bound and a lower bound on the rate function. In most parts of the proof we will see that one bound is straightforward  and for the other bound we decompose our LPM-BRW at an intermediate generation. In fact, the strategy for proving (iv) and (v) is  similar to that used in Gantert and  H\"{o}felsauer~\cite{GH18}. To explain the decomposition more formally, for $v$ such that $|v|=m\leq n$, we define
  \begin{equation}
  	R_{n-m}^{*(v)}:=\max_{|u|=n,v<u}\Big\{S(u)+\frac{1}{\theta}\log (Y_u/E_u)\Big\}-S(v).
  	\label{proofprelim:eq1}
  \end{equation}
Here $v<u$ means $u$ is a descendant of $v$. Note that $\{R_{n-m}^{*(v)}\}_{|v|=m}$ are i.i.d. copies of $R_{n-m}^*$ and are independent of the BRW up to generation $m$. Now,~\eqref{def:rnrn*} implies that
\begin{align}
	R_{n}^* = \max_{|v|=m}\Big\{\max_{ |u|=n,v<u }\Big\{S(u)+\frac{1}{\theta}\log (Y_u/E_u)\Big\}\Big\}
	&=\max_{|v|=m}\Big\{S(v)+R_{n-m}^{*(v)}\Big\}\nonumber\\
	&\geq S({\tilde{v}_m})+\max_{|v|=m}R_{n-m}^{*(v)}, \label{stineq}
\end{align}	
where $\tilde{v}_m:=\arg\max_{|v|=m} R_{n-m}^{*(v)}$.	
Since $\{S(v)\}_{|v|=m}$ are identically distributed and are independent of $\{R_{n-m}^{*(v)}\}_{|v|=m}$, we have
\begin{equation}
	S({\tilde{v}_m}) \xlongequal{d} \mathsf{S}_{m}.
	\label{disteq:S}
\end{equation}

To prove Theorem~\ref{ratefn}, we also need the following lemma, which provides LDP
for each of the branches of the LPM-BRW. 

\begin{lemma}
	\label{ld1}
	Let $Y\sim\mu$ and $E\sim\mbox{Exponential}\left(1\right)$ be independent of each other and also independent of the random variables $\{X_n\}_{n\geq1}$. Then, for any $\theta>0$, the laws of $\big\{\frac{\mathsf{S}_n}{n}+\frac{1}{n\theta}\log (Y/E)\big\}_{n\geq1}$ satisfy the large deviation principle with the rate function
	\[I_{\theta}(x):=
	\left\{
	\begin{array}{ll}
		I(x) ,&\text{if $x\leq\phi'(\theta)$;}\\
		\theta x-\phi(\theta),&\text{if $x\geq\phi'(\theta)$.}
	\end{array}
	\right.
	\]
\end{lemma}
\begin{proof}
	For each $\theta>0$ and $\lambda\in \R$, we define 
	\begin{align*}
		\Upsilon_{\theta}(\lambda):&=\lim_{n\rightarrow\infty}\frac{1}{n}\log\E\big[e^{\lambda\mathsf{S}_n+\frac{\lambda}{\theta}\log (Y/E)}\big]\\
		&=\lim_{n\rightarrow\infty}\frac{1}{n}\log\Big(e^{n\phi(\lambda)}\cdot\E\big[Y^{\lambda/\theta}\big]\cdot\E\big[E^{-\lambda/\theta}\big]\Big)=\left\{
		\begin{array}{ll}
			\phi(\lambda),&\text{if $\lambda<\theta$;}\\
			\infty,&\text{if $\lambda\geq\theta$.}
		\end{array}
		\right.
	\end{align*}
	Its Fenchel-Legendre transform is
	\begin{align*}
		\Upsilon_{\theta}^*(x):=\sup_{\lambda\in\R} \left\{\lambda x-\Upsilon_{\theta}(\lambda)\right\}
		=\sup_{\lambda<\theta} \left\{\lambda x-\phi(\lambda)\right\} =I_{\theta}(x).
	\end{align*} 
	Since $0$ belongs to the interior of the set $\left\{\lambda\in\R:\Upsilon_{\theta}(\lambda)<\infty \right\}$, it follows
	from the G\"{a}rtner-Ellis theorem (see Dembo and Zeitouni~\cite{DZ10}) that for any closed set $F$,
	\begin{equation}
		\limsup_{n\rightarrow\infty}\frac{1}{n}\log\P\Big(\frac{\mathsf{S}_n}{n}+\frac{1}{n\theta}\log (Y/E)\in F\Big)
		\leq -\inf_{x\in F} I_{\theta}(x), \label{itheta1}
	\end{equation}
	and for any open set $G$,
	\begin{equation}
		\liminf_{n\rightarrow\infty}\frac{1}{n}\log\P\Big(\frac{\mathsf{S}_n}{n}+\frac{1}{n\theta}\log (Y/E)\in G\Big)
		\geq -\inf_{x\in G,\, x<\phi'(\theta)} I_{\theta}(x). \label{itheta2}
	\end{equation}
	Note that since $Y$ is a positive random variable, there exists $\alpha>0$ such that $\P(Y>\alpha)>0$.
	Now, for any $x\geq\phi'(\theta)$, we have
	\begin{align*}
		\P\Big( \frac{\mathsf{S}_n}{n}+\frac{1}{n\theta}\log (Y/E)> x \Big)
		\geq \P\big( \mathsf{S}_n> n\phi'(\theta) \big) \cdot\P(Y>\alpha)\cdot\P\big( E<\alpha e^{-n\theta\left(x-\phi'(\theta)\right)} \big).
	\end{align*}
	Therefore using  Cram\'{e}r's theorem, we get
	\begin{align}
		\liminf_{n\rightarrow\infty} \frac{1}{n}\log \P\Big( \frac{\mathsf{S}_n}{n}+\frac{1}{n\theta}\log (Y/E)> x \Big)
		\geq  -I\left( \phi'(\theta)\right)-\theta\left(x-\phi'(\theta)\right)=-I_{\theta}(x). \label{itheta3}
	\end{align}
	Combining~\eqref{itheta2} and~\eqref{itheta3}, we obtain that for any open set $G$,
	\begin{equation*}
		\liminf_{n\rightarrow\infty}\frac{1}{n}\log\P\Big(\frac{\mathsf{S}_n}{n}+\frac{1}{n\theta}\log (Y/E)\in G\Big)
		\geq -\inf_{x\in G} I_{\theta}(x). 
	\end{equation*}
	This, together with~\eqref{itheta1}, completes the proof.
\end{proof}
\vspace{\baselineskip}

  Now we have all the machinery to prove Theorem~\ref{ratefn}. 

\subsubsection{Proof of (vi)}\label{proofvi}
\begin{proof}
	Take any $x<c(\theta)$ and $\epsilon\in\left(0,c(\theta)-x\right)$. Using inequality~\eqref{stineq}, we have
	\begin{align}
		\P(R_n^*<nx)
		&\leq \P\Big(S({\tilde{v}_{\lfloor\sqrt{n}\rfloor}})+\max_{|v|=\lfloor\sqrt{n}\rfloor}R_{n-\lfloor\sqrt{n}\rfloor}^{*(v)}<nx\Big)\nonumber \\[.15cm]
		&\leq
		\P\Big(\max_{|v|=\lfloor\sqrt{n}\rfloor}R_{n-\lfloor\sqrt{n}\rfloor}^{*(v)}<n(x+\epsilon)\Big)
		+ \P\big(S({\tilde{v}_{\lfloor\sqrt{n}\rfloor}})<-n\epsilon\big)\nonumber\\[.15cm]
		&\leq
		\E\Big[\P\Big(R_{n-\lfloor\sqrt{n}\rfloor}^*<n(x+\epsilon)\Big)^{N_{\lfloor\sqrt{n}\rfloor}}\Big]+
		\P\big(\mathsf{S}_{\lfloor\sqrt{n}\rfloor}<-n\epsilon\big). \label{proof6:eq1}
	\end{align}
	Here $\lfloor x\rfloor$ denotes the greatest integer less than or equal to $x$, and $N_k$ represents the total number of particles at generation $k$. Note that $N_k$ is at least $2^k$ since $\P(N=1)=0$. Now since $x+\epsilon<c(\theta)$, which is the almost sure limit of $R_{n-\lfloor\sqrt{n}\rfloor}^*/n$, we have
	\begin{align}
		&\limsup_{n\rightarrow\infty}
		\frac{1}{n}\log  \E\Big[\P\Big(R_{n-\lfloor\sqrt{n}\rfloor}^*<n(x+\epsilon)\Big)^{N_{\lfloor\sqrt{n}\rfloor}}\Big]\nonumber \\[.15cm]
		\leq\,\,&\lim_{n\rightarrow\infty}
				\frac{2^{\lfloor\sqrt{n}\rfloor}}{n}\log\P\big(R_{n-\lfloor\sqrt{n}\rfloor}^*<n(x+\epsilon)\big)
				=-\infty.
		\label{proof6:eq2}
	\end{align}
	Let $\{t_n\}_{n\geq1}$ be a non-negative real sequence increasing to $\infty$ such that $\phi(-t_n)\leq \log n$. Such a sequence exists since $\phi(\lambda)<\infty$ for all $\lambda\leq0$. Then using  Markov's inequality we obtain
	\begin{align}
			\limsup_{n\rightarrow\infty}\frac{1}{n}\log\P\big(\mathsf{S}_{\lfloor\sqrt{n}\rfloor}<-n\epsilon\big)
			&\leq \lim_{n\rightarrow\infty}  \frac{1}{n}\log \big(e^{-nt_n\epsilon}\cdot \E\big[e^{-t_n\mathsf{S}_{\lfloor\sqrt{n}\rfloor}}\big]\big)\nonumber\\[.15cm]
			&= \lim_{n\rightarrow\infty} -t_n\epsilon + \frac{\lfloor\sqrt{n}\rfloor\phi(-t_n)}{n}
			=-\infty.
		\label{proof6:eq3}
	\end{align}
	Therefore, by combining~\eqref{proof6:eq1},~\eqref{proof6:eq2}, and~\eqref{proof6:eq3}, we get that for $\rho=\infty$ and all $x<c(\theta)$,
	\begin{equation}
		\lim_{n\rightarrow\infty}
		-\frac{1}{n}\log \P(R_n^*<nx)=\infty.
		\label{proved6}
	\end{equation}
\end{proof}

\subsubsection{Proof of (iv) \& (v)}\label{proofivv}

\begin{proof}
	\textit{\textbf{(Lower bound).}}
	Take any $x<c(\theta)$ and $t\in(0,1]$. Observe that for $|v|=\lceil tn \rceil$ and $\epsilon>0$,
	\begin{align}
		\P(R_n^*<nx) &\geq \P\big( S(v) +R_{\lfloor(1-t)n\rfloor}^{*(v)}<nx, N_{\lceil tn \rceil}=1  \big) \nonumber\\
		&\geq \P\big(N_{\lceil tn \rceil}=1\big)\cdot\P\big(R_{\lfloor(1-t)n\rfloor}^*<n(1-t)(c(\theta)+\epsilon)\big)\nonumber\\
		&\qquad\cdot\P\big(\mathsf{S}_{\lceil tn \rceil}<nx-n(1-t)(c(\theta)+\epsilon)\big).
		\label{proof45:lb:eq1}
	\end{align}
	Here $\lceil x \rceil$ denotes the smallest integer greater than or equal to $x$. Also, note that $N_{\lceil tn \rceil}$, $S(v)$ and $R_{\lfloor(1-t)n\rfloor}^{*(v)}$ are independent of each other, which implied the last inequality.
	For the first term on the right-hand side, we have
	\begin{equation}
		\lim_{n\rightarrow\infty}\frac{1}{n}\log \P\big(N_{\lceil tn \rceil}=1\big)=\lim_{n\rightarrow\infty}\frac{1}{n}\log \P(N=1)^{\lceil tn \rceil}=-\rho t.\label{ivvlb1}
	\end{equation}
	For $t=1$, the second term equals $\P(Y<E)>0$,
	 and for $t\in(0,1)$,  $c(\theta)$ is the almost sure limit of $R_{\lfloor(1-t)n\rfloor}^*/\left(n(1-t)\right) $. Therefore for all $t\in(0,1]$, we have
	\begin{equation}
		\lim_{n\rightarrow\infty}\frac{1}{n}\log \P\big(R_{\lfloor(1-t)n\rfloor}^*<n(1-t)(c(\theta)+\epsilon)\big)=0.\label{ivvlb2}
	\end{equation}
	Finally, for the last term, using  Cram\'{e}r's theorem, we get
	\begin{align}
		\lim_{n\rightarrow\infty}\frac{1}{n}\log\P\big(\mathsf{S}_{\lceil tn \rceil}<nx-n(1-t)(c(\theta)+\epsilon)\big)
		= -tI\bigg(\frac{x-(1-t)(c(\theta)+\epsilon)}{t}\bigg), \label{ivvlb3}
	\end{align}
	whenever 
	\[
	0<t\leq f(x):=\min\Big\{1, \frac{c(\theta)-x}{c(\theta)-\E[X]}  \Big\}.
	\]
	So, by combining~\eqref{proof45:lb:eq1},~\eqref{ivvlb1},~\eqref{ivvlb2}, and~\eqref{ivvlb3}, and allowing $\epsilon\downarrow 0$, we obtain
	\[
	\liminf_{n\rightarrow\infty}\frac{1}{n}\log \P(R_n^*<nx)
	\geq -\inf\limits_{0<t\leq f(x)}\Big\{\rho t+tI\Big(\frac{x-(1-t)c(\theta)}{t}\Big)\Big\}.
	\]
	Since $ I\big( \big(x-(1-t)c(\theta)\big)/t \big) $ is non-decreasing for
	$t\geq  {\big(c(\theta)-x\big)}/{\big(c(\theta)-\E[X]\big)}$,
	the above inequality implies
	\begin{equation}
		\liminf_{n\rightarrow\infty}\frac{1}{n}\log \P(R_n^*<nx)
		\geq -\inf\limits_{0<t\leq1}\Big\{\rho t+tI\Big(\frac{x-(1-t)c(\theta)}{t}\Big)\Big\}.
		\label{ng1}
	\end{equation}
	
	\textit{\textbf{(Upper bound).}}
	Now, we fix any $k\in\N$ and define 
	$n_i=\left\lfloor{nif(x)}/{k}\right\rfloor$
	for all $i=0,1,2,\ldots,k$. Since $N_{{n_0}}=N_0=1$, for any $n\geq2$, we have
	\begin{align}
		\P(R_n^*<nx)
		&=\sum_{i=0}^{k-2}\P(N_{n_i}<n^2,N_{n_{i+1}}\geq n^2) \cdot\P(\left.R_n^*<nx\right| N_{n_i}<n^2,N_{n_{i+1}}\geq n^2)\nonumber\\
		&\qquad\qquad+\P(N_{{n_{k-1}}}<n^2)\cdot\P( \left.R_n^*<nx\right| N_{{n_{k-1}}}<n^2). \label{ivvub0}
	\end{align}
	Using Theorem~2.5 of Gantert and  H\"{o}felsauer~\cite{GH18}, we get that for $1\leq i\leq k-1$,
	\begin{align}
		\lim_{n\rightarrow\infty}\frac{1}{n}\log\P(N_{{n_i}}<n^2  ) =-\frac{if(x)\rho}{k}.\label{ivvub1}
	\end{align}
	On the other hand, using inequality~\eqref{stineq}, we have for all $\epsilon>0$ and $0\leq i\leq k-2$,
	\begin{align}
		&\P\big( \left.R_n^*<nx \right| N_{n_i}<n^2 ,N_{n_{i+1}}\geq n^2 \big)\nonumber\\[.15cm]
		\leq\,\, & \P\big( \left.S({\tilde{v}_{n_{i+1}}})<nx-(n-{n_{i+1}})(c(\theta)-\epsilon)\right| N_{{n_i}}<n^2,N_{n_{i+1}}\geq n^2\big)\nonumber\\[.15cm]
		&\qquad+\P\Big(\Big.\max_{|v|={n_{i+1}}}R^{*(v)}_{n-{n_{i+1}}}<(n-{n_{i+1}})(c(\theta)-\epsilon)\Big| N_{{n_i}}<n^2,N_{n_{i+1}}\geq n^2\Big)\nonumber\\[.15cm]
		\leq\,\,  & \P\big(\mathsf{S}_{n_{i+1}}<nx-(n-{n_{i+1}})(c(\theta)-\epsilon)\big)+ \P\big(R^*_{n-{n_{i+1}}}<(n-{n_{i+1}})(c(\theta)-\epsilon)\big)^{n^2}.\label{ivvub2}
	\end{align}
	Notice that $S({\tilde{v}_{n_{i+1}}})$ is independent of $N_{n_{i}}$ and $N_{n_{i+1}}$, and by~\eqref{disteq:S}, it has the same distribution as $\mathsf{S}_{n_{i+1}}$,	which implied the last inequality.  
	Now, 
	we know that $c(\theta)$ is the almost sure limit of $R^*_{n-{n_{i+1}}}/ (n-{n_{i+1}})$. Therefore  for any $i=0,1,2,\ldots, k-2$,
	\[\lim_{n\rightarrow\infty}\frac{1}{n}\log\P\big(R^*_{n-{n_{i+1}}}<(n-{n_{i+1}})(c(\theta)-\epsilon)\big)^{n^2}= -\infty. \]
	Thus, from~\eqref{ivvub2}, we get that  for any $i=0,1,2,\ldots, k-2$ and $\epsilon>0$ small enough,
	\begin{align}
		&\limsup_{n\rightarrow\infty}\frac{1}{n}\log\P\big( \big. R_n^*<nx\big| N_{{n_i}}<n^2,N_{{n_{i+1}}}\geq n^2\big)\nonumber\\[.15cm]
		\leq& 
		\lim_{n\rightarrow\infty}\frac{1}{n}\log\P\left(\mathsf{S}_{{n_{i+1}}}<nx-(n-{n_{i+1}})(c(\theta)-\epsilon)\right)\nonumber\\[.15cm]
		=&- \frac{(i+1)f(x)}{k}\cdot I\bigg(\frac{x-\big(1-\frac{(i+1)f(x)}{k}\big)\left(c(\theta)-\epsilon\right)}{\frac{(i+1)f(x)}{k}}\bigg).
		\label{ivvub3}
	\end{align}
	For the last term of~\eqref{ivvub0}, if $f(x)=\left({c(\theta)-x}\right)/\left({c(\theta)-\E[X]}\right)$, we trivially have
	\begin{align*}
		\limsup_{n\rightarrow\infty}\frac{1}{n}\log\P\big( \big. R_n^*<nx\big| N_{{n_{k-1}}}<n^2\big) 
		\leq 0 = -f(x)\cdot I\big(\E[X]\big) .
	\end{align*}
	and if  $f(x)=1$, we have $x\leq \E[X]$. In that case, from	Lemma~\ref{ld1}, we have
	\begin{align*}
		\limsup_{n\rightarrow\infty}\frac{1}{n}\log \P\big( \big. R_n^*<nx \big| N_{{n_{k-1}}}<n^2\big)
		\leq \limsup_{n\rightarrow\infty}\frac{1}{n}\log\P\Big(\mathsf{S}_n+\frac{1}{\theta}\log (Y/E)<nx\Big)
		=-I(x).
	\end{align*}
Combining the above two inequalities, we get
\begin{align}
	\limsup_{n\rightarrow\infty}\frac{1}{n}\log \P\big( \big. R_n^*<nx \big| N_{{n_{k-1}}}<n^2\big) \leq -f(x)\cdot I\Big(\frac{x-(1-f(x))c(\theta)}{f(x)}\Big).
	\label{ivvub4}
\end{align}
	Threfore, by combining~\eqref{ivvub0},~\eqref{ivvub1},~\eqref{ivvub3}, and~\eqref{ivvub4}, and then allowing $\epsilon\downarrow 0$ and $k\rightarrow\infty$,  we obtain 
	\begin{equation}
		\limsup_{n\rightarrow\infty}\frac{1}{n}\log \P(R_n^*<nx)
		\leq -\inf\limits_{0<t\leq1}\Big\{\rho t+tI\Big(\frac{x-(1-t)c(\theta)}{t}\Big)\Big\}.
		\label{ng2}
	\end{equation}
	This, together with~\eqref{ng1}, implies that for any $x<c(\theta)$ and $\rho<\infty$,
	\begin{align}
		\lim_{n\rightarrow\infty}-\frac{1}{n}\log \P(R_n^*<nx)
		&=\inf\limits_{0<t\leq1}\Big\{\rho t+tI\Big(\frac{x-(1-t)c(\theta)}{t}\Big)\Big\}\nonumber\\[.15cm]
		&=\inf\limits_{y\leq x}\Big\{\left(\rho+I(y)\right)\frac{c(\theta)-x}{c(\theta)-y}  \Big\}\nonumber\\[.15cm]
		&=\left(c(\theta)-x\right)\Big(\inf\limits_{y\leq x}\Big\{ \frac{\rho+I(y)}{c(\theta)-y} \Big\}\Big)\nonumber\\[.15cm]
		&=\left\{
		\begin{array}{ll}
			I'\left(a_{\theta}^{\rho}\right)\left(x-c(\theta)\right), &\text{if $a_{\theta}^{\rho}\leq x< c(\theta);$}\\
			I(x)+\rho, &\text{if $x< a_{\theta}^{\rho}.$}
		\end{array}
		\right.\label{gantertsimplify}
	\end{align}
\end{proof}

\subsubsection{Proof of (iii)}\label{proofiii}
\begin{proof}
	This part follows from~\eqref{conv:rn*}.
\end{proof}

\subsubsection{Proof of (ii)}\label{proofii}
\begin{proof}
	Note that $c(\theta)< d(\theta)$ means $d(\theta)=\phi'(\theta)=\nu'(\theta)$. Therefore, $c(\theta)< d(\theta)$ occurs iff $\theta_0<\infty$ and $\theta>\theta_0$. So this part is only relevant for this range of $\theta$. 
	
	\textit{\textbf{(Upper bound).}}
	Take any $x\in(c(\theta),d(\theta)]$ and observe that
	\begin{align*}
		\P(R_{n}^*>nx)&=\E\big[\P(R_{n}^*>nx|N_n) \big]\\[.15cm]
		&\leq \E\Big[N_n\cdot\P\Big(\mathsf{S}_n+\frac{1}{\theta}\log (Y/E)>nx\Big) \Big]\\[.15cm]
		&=\big(\E[N]\big)^n\cdot\P\Big(\mathsf{S}_n+\frac{1}{\theta}\log (Y/E)>nx\Big).
	\end{align*}
	Since $\phi'(\theta)\geq x>c(\theta)>\E[X]$, using Lemma~\ref{ld1},  we get
	\begin{equation}
		\limsup_{n\rightarrow\infty} \frac{1}{n}\log\P(R_{n}^*>nx)\leq \log\E[N]-I(x)\label{iiub}.
	\end{equation}
	
	\textit{\textbf{(Lower bound).}}	
	For any $\alpha\in(0,1)$, using inequality~\eqref{stineq}, we obtain
	\begin{align*}
		\P(R_n^*>nx)\,
		&\geq \P\Big(S({\tilde{v}_{\lfloor\alpha{n}\rfloor}})+\max_{|v|=\lfloor\alpha{n}\rfloor}R_{\lceil(1-\alpha)n\rceil}^{*(v)}>nx\Big) \\[.15cm]
		&\geq \P\big(S({\tilde{v}_{\lfloor\alpha{n}\rfloor}})>\lfloor\alpha{n}\rfloor x\big)\cdot
		\P\Big(\max_{|v|=\lfloor\alpha{n}\rfloor} R_{\lceil(1-\alpha)n\rceil}^{*(v)}>\lceil(1-\alpha)n\rceil x\Big)\\[.15cm]
		&\geq  \P\big(\mathsf{S}_{\lfloor\alpha{n}\rfloor}>\lfloor\alpha{n}\rfloor x\big)\cdot
		\P\Big(N_{\lfloor\alpha n\rfloor}>\frac{1}{2}\cdot\E[N]^{\lfloor\alpha n\rfloor}\Big)\\[.15cm]
		&\qquad\qquad\cdot\P\Big( \Big. \max_{|v|=\lfloor\alpha{n}\rfloor}R_{\lceil(1-\alpha)n\rceil}^{*(v)}>\lceil(1-\alpha)n\rceil x \Big| N_{\lfloor\alpha n\rfloor}>\frac{1}{2}\cdot\E[N]^{\lfloor\alpha n\rfloor} \Big)\\[.15cm]
		&\geq  \P\big(\mathsf{S}_{\lfloor\alpha{n}\rfloor}>\lfloor\alpha{n}\rfloor x\big)\cdot
		\P\Big(N_{\lfloor\alpha n\rfloor}>\frac{1}{2}\cdot\E[N]^{\lfloor\alpha n\rfloor}\Big)\\[.15cm]
		&\qquad\qquad\cdot\Big(1-\Big(1-\P\big( R_{\lceil(1-\alpha)n\rceil}^{*}>\lceil(1-\alpha)n\rceil x   \big)     \Big)^{\frac{1}{2}\cdot\E[N]^{\lfloor\alpha n\rfloor}}     \Big).
	\end{align*}	
	For any $a\in[0,1]$ and $t\geq2$, we know that $1-(1-a)^t \geq at(1-at)$.
	Therefore, for all large enough $n$, we get
	\begin{align}
		\P(R_n^*>nx)
		&\geq  \P\big(\mathsf{S}_{\lfloor\alpha n\rfloor}>\lfloor\alpha n\rfloor x\big)\cdot
		\P\Big(N_{\lfloor\alpha n\rfloor}>\frac{1}{2}\cdot\E[N]^{\lfloor\alpha n\rfloor}\Big)\nonumber\\[.15cm]
		&\qquad\cdot\frac{1}{2}\cdot\E[N]^{\lfloor\alpha n\rfloor}\cdot\P\big( R_{\lceil(1-\alpha)n\rceil}^{*}>\lceil(1-\alpha)n\rceil x   \big)\nonumber\\[.15cm]
		&\qquad\qquad\cdot\Big(1-\frac{1}{2}\cdot\E[N]^{\lfloor\alpha n\rfloor}\cdot\P\big( R_{\lceil(1-\alpha)n\rceil}^{*}>\lceil(1-\alpha)n\rceil x   \big)\Big).\label{iilb1}
	\end{align}	
	Note that since $c(\theta)=\nu(\theta_0)/\theta_0$, we have
	\begin{align*}
		I(x)=\sup_{\lambda\in\R} \left\{\lambda x-\phi(\lambda) \right\}\geq \theta_0 x -\phi(\theta_0) 
		= \theta_0\left(x -c(\theta)\right)+\log\E[N].
	\end{align*}
	Now, for all $x\in(c(\theta),d(\theta)]$, we choose $\alpha_x$ such that 
	\[0<\alpha_x<\frac{\theta_0\left(x -c(\theta)\right)}{\theta_0\left(x -c(\theta)\right)+\log\E[N]},\]
	which ensures $ (1-\alpha_x)I(x)> \log\E[N]$. 
	Together with~\eqref{iiub}, this implies 
	\[
	\lim_{n\rightarrow\infty}\E[N]^{ \lfloor\alpha_x n\rfloor}\cdot\P\big( R_{\lceil(1-\alpha_x)n\rceil}^{*}>\lceil(1-\alpha_x)n\rceil x   \big)=0.
	\]
	Therefore, for $\alpha=\alpha_x$, the last term on the right-hand side of~\eqref{iilb1} tends to 1, as $n$ tends to $\infty$.	
	Also, 
	assumption (\hyperlink{A2}{$\mathbf{A2}$})  implies that (see Athreya and Ney~\cite{AN72})
	\[
	\lim_{n\rightarrow\infty}\P\Big(N_{\lfloor\alpha_x n\rfloor}>\frac{1}{2}\cdot\E[N]^{\lfloor\alpha_x n\rfloor}\Big)>0.
	\]
	Thus inequality~\eqref{iilb1} indicates
	\begin{align*}
		\liminf_{n\rightarrow\infty}\frac{1}{n}\log\P(R_n^*>nx)
		& \geq
		\lim_{n\rightarrow\infty}\frac{1}{n}\log \P\big(\mathsf{S}_{\lfloor\alpha_x n\rfloor}>\lfloor\alpha_x n\rfloor x\big)+
		\lim_{n\rightarrow\infty}\frac{1}{n}\log \E[N]^{\lfloor\alpha_x n\rfloor}\\[.15cm]
		&\qquad\qquad+\liminf_{n\rightarrow\infty}\frac{1}{n}\log \P\big( R_{\lceil(1-\alpha_x)n\rceil}^{*}>\lceil(1-\alpha_x)n\rceil x   \big).
	\end{align*}
	Together with  Cram\'{e}r's theorem, this implies 
	\begin{align}
		\liminf_{n\rightarrow\infty}\frac{1}{n}\log\P(R_n^*>nx)
		\geq
		\alpha_x\big(\log\E[N]-I(x)\big)
		+(1-\alpha_x)\liminf_{n\rightarrow\infty}\frac{1}{n}\log\P(R_n^*>nx).\label{iilb2}
	\end{align}
	Since $I(x)$ is finite for $x\in\left(c(\theta),d(\theta)\right]=\left(\phi'(\theta_0),\phi'(\theta)\right]$,   using Lemma~\ref{ld1}, we  have
	\begin{align*}
		\liminf_{n\rightarrow\infty}\frac{1}{n}\log\P(R_n^*>nx)
		\geq\lim_{n\rightarrow\infty}\frac{1}{n}\log\P\Big(\mathsf{S}_n+\frac{1}{\theta}\log (Y/E)>nx\Big)
		=-I(x)
		>-\infty.
	\end{align*}
	So,  from~\eqref{iilb2}, we get
	\begin{equation}
		\liminf_{n\rightarrow\infty}\frac{1}{n}\log\P(R_n^*>nx)\geq \log\E[N]-I(x)\label{iilb3}.
	\end{equation}
	Combining~\eqref{iiub} and~\eqref{iilb3},  we obtain that for all $x\in \left(c(\theta), d(\theta)\right]$,
	\begin{equation}
		\lim_{n\rightarrow\infty}-\frac{1}{n}\log\P\left(R_n^*>nx\right)= I(x)-\log\E[N].
		\label{proved2}
	\end{equation}
\end{proof}

\subsubsection{Proof of (i)}\label{proofi}

\begin{proof}
	\textit{\textbf{(Upper bound).}}
	Using  Markov's inequality, we obtain that for any $x\in\R$ and any $\lambda<\theta$,
	\begin{align*}
		\P(R_n^*>nx) \leq e^{-n\lambda x}\cdot\E\big[e^{\lambda R_n^*}\big]
		&\leq e^{-n\lambda x}\cdot\E\Big[\sum_{|v|=n}e^{\lambda S(v)}Y_v^{\lambda/\theta}E_v^{-\lambda/\theta}\Big]\\
		&= e^{-n\lambda x}\cdot\E[N]^n\cdot e^{n\phi(\lambda)}\cdot\E\big[Y^{\lambda/\theta}\big]\cdot\Gamma\Big(1-\frac{\lambda}{\theta}\Big).
	\end{align*}
	Since this inequality holds for all $\lambda<\theta$, we have
	\begin{align}
		\limsup_{n\rightarrow\infty} \frac{1}{n}\log\P(R_n^*>nx) &\leq \lim_{\lambda\uparrow\theta}-\lambda x + \phi(\lambda) +\log\E[N] =-\theta x + \phi(\theta) +\log\E[N]. \label{iub} 
	\end{align}
	
	\textit{\textbf{(Lower bound).}}
	For every positively supported probability $\eta$, we define 
	\[
	A_n^{\eta}(\theta):=\sum_{|v|=n}e^{\theta S(v)}Z_v,
	\]
	where $\{Z_v\}_{|v|=n}$ are i.i.d. $\eta$ and are independent of the BRW. Bandyopadhyay and Ghosh~\cite{BG21} showed that 
	\begin{equation}
		\theta R_n^*(\theta,\eta)\xlongequal{d} \log A_n^{\eta}(\theta) - \log E, \label{rn*an}
	\end{equation}
	where $E\sim\mbox{Exponential}\left(1\right)$ and is independent of $\{Z_v\}_{|v|=n}$ and also of the BRW.
	Therefore  we get that for any $x>d(\theta)$ and any $\epsilon>0$,
	\begin{align}
		\P\Big(\frac{R_n^*(\theta,\mu)}{n}>x\Big)
		&=\P\Big(\frac{\log A_n^{\mu}(\theta)}{n\theta}-\frac{\log E}{n\theta}>x\Big)\nonumber\\[.25cm]
		&\geq \P\Big(\frac{\log A_n^{\mu}(\theta)}{n\theta}>d(\theta)-2\epsilon\Big)\cdot\P\Big(-\frac{\log E}{n\theta}>x-d(\theta)+2\epsilon\Big).
		\label{ilb1}
	\end{align}
	Now, take any $\theta_1\geq\theta$ and denote $\mu_1$ as the distribution of $Y^{\theta_1/\theta}$. Then we have
	\begin{align}
		\big(A_n^{\mu}(\theta)\big)^{1/\theta}=\Big(\sum_{|v|=n}e^{\theta S(v)}Y_v\Big)^{1/\theta}\geq
		\Big(\sum_{|v|=n}e^{\theta_1 S(v)}Y_v^{\theta_1/\theta}\Big)^{1/\theta_1}
		=\big(A_n^{\mu_1}(\theta_1)\big)^{1/\theta_1}. \label{ilb2}
	\end{align}
	From~\eqref{rn*an}, we also have
	\begin{align}
		\P\Big(\frac{R_n^*(\theta_1,\mu_1)}{n}>d(\theta)-\epsilon\Big)
		&=\P\Big(\frac{\log A_n^{\mu_1}(\theta_1)}{n\theta_1}-\frac{\log E}{n\theta_1}>d(\theta)-\epsilon\Big)\nonumber\\[.25cm]
		&\leq \P\Big(\frac{\log A_n^{\mu_1}(\theta_1)}{n\theta_1}>d(\theta)-2\epsilon\Big)+\P\Big(-\frac{\log E}{n\theta_1}>\epsilon\Big). \label{ilb3}
	\end{align}
	Therefore, by combining~\eqref{ilb1},~\eqref{ilb2}, and~\eqref{ilb3}, we obtain
	\begin{align}
		\P\Big(\frac{R_n^*(\theta,\mu)}{n}>x\Big)
		&\geq
		\Big( \P\Big(\frac{R_n^*(\theta_1,\mu_1)}{n}>d(\theta)-\epsilon\Big)- \P\Big(-\frac{\log E}{n\theta_1}>\epsilon\Big)  \Big)\nonumber\\
		&\qquad\cdot\P\Big(-\frac{\log E}{n\theta}>x-d(\theta)+2\epsilon\Big). \label{ilb4}
	\end{align}
	Observe that for any $t>0$,
	\begin{align}
		\lim_{n\rightarrow\infty}\frac{1}{n}\log\P(-\log E>nt)
		=\lim_{n\rightarrow\infty}\frac{1}{n}\log\big(1-e^{-e^{-nt}}\big) 	=-t. \label{ilb5}
	\end{align}
	Now, for $\theta<\theta_0$ or $\theta=\theta_0<\infty$, we take $\theta_1=\theta$. In that case, $d(\theta)=c(\theta)$, which implies
	\[\lim_{n\rightarrow\infty} \frac{1}{n}\log \P\Big(\frac{R_n^*(\theta_1,\mu_1)}{n}>d(\theta)-\epsilon\Big) =0.\]
	As a result, in view of~\eqref{ilb4} and~\eqref{ilb5}, we get that for $\theta<\theta_0$ or $\theta=\theta_0<\infty$,
	\begin{equation}
		\liminf_{n\rightarrow\infty} \frac{1}{n}\log \P\Big(\frac{R_n^*(\theta,\mu)}{n}>x\Big)\geq -\theta\left(x-d(\theta)+2\epsilon \right). \label{ilb6}
	\end{equation}
	For $\theta_0<\theta<\infty$, we know that $c(\theta)<d(\theta)$. So choosing $\epsilon<d(\theta)-c(\theta)$, by part (ii) of the theorem, we have
	\[\lim_{n\rightarrow\infty} \frac{1}{n}\log \P\Big(\frac{R_n^*(\theta_1,\mu_1)}{n}>d(\theta)-\epsilon\Big) =-\Psi_{\theta_1}\left(d(\theta)-\epsilon\right)=-\Psi_{\theta}\left(d(\theta)-\epsilon\right).\]
	Now, we choose $\theta_1$ large enough such that
	$\theta_1\epsilon>\Psi_{\theta}\left(d(\theta)-\epsilon\right)$,
	which ensures
	\[\lim_{n\rightarrow\infty} \frac{1}{n}\log \Big(\P\Big(\frac{R_n^*(\theta_1,\mu_1)}{n}>d(\theta)-\epsilon\Big)
	-\P\Big(-\frac{\log E}{n\theta_1}>\epsilon\Big)
	\Big) =-\Psi_{\theta}\left(d(\theta)-\epsilon\right).\]
	Together with~\eqref{ilb4} and~\eqref{ilb5}, this implies that for $\theta_0<\theta<\infty$,
	\begin{equation}
		\liminf_{n\rightarrow\infty} \frac{1}{n}\log \P\Big(\frac{R_n^*(\theta,\mu)}{n}>x\Big)\geq -\Psi_{\theta}\left(d(\theta)-\epsilon\right)
		-\theta\left(x-d(\theta)+2\epsilon \right).
		\label{ilb6.5}
	\end{equation}
	Since $\epsilon>0$ can be chosen arbitrarily small and $\Psi_{\theta}$ is continuous in $[c(\theta),\infty)$, by combining~\eqref{ilb6} and~\eqref{ilb6.5}, we get that for any $\theta>0$,
	\begin{equation}
		\liminf_{n\rightarrow\infty} \frac{1}{n}\log \P\Big(\frac{R_n^*(\theta,\mu)}{n}>x\Big)\geq
		-\Psi_{\theta}\left(d(\theta)\right)
		-\theta\left(x-d(\theta) \right)
		= -\Psi_{\theta}\left(x\right). \label{ilb7}
	\end{equation}
	Thus, by combining~\eqref{iub} and~\eqref{ilb7}, we finally obtain that for any $x>d(\theta)$,
	\begin{equation}
		\lim_{n\rightarrow\infty} -\frac{1}{n}\log\P\left(R_n^*>nx\right) = \theta x - \phi(\theta) -\log\E[N].
		\label{proved1}
	\end{equation}
\end{proof}

\subsection{Proof of Theorem~\ref{ratefnrn}}
For (iii) and (iv), the expression in Gantert and  H\"{o}felsauer~\cite{GH18} can be simplified as we did in equation~\eqref{gantertsimplify}.
The proof of (v) is essentially the proof of the part \hyperref[proofvi]{(vi)} of Theorem~\ref{ratefn} verbatim.
Note that the assumption $\E[N^{1+p}]<\infty$  in (\hyperlink{A2}{$\mathbf{A2}$}) was only required for the almost sure convergence of $R_n^*/n$ and therefore is not required to prove part (v) of Theorem~\ref{ratefnrn}. But we do need $\E[N \log N] < \infty$ for the remaining parts, as shown in Gantert and  H\"{o}felsauer~\cite{GH18}.

\section{Comparision with Branching Random Walk}
\label{sec:comparision}
We observe that for $\theta_0\leq \theta<\infty$, the lower large deviations for the laws of $\{R_n/n\}_{n\geq1}$ and $\{R_n^*(\theta,\mu)/n\}_{n\geq1}$  coincide. It should be noted that there is an error in deriving the lower large deviations for the laws of $\{R_n/n\}_{n\geq1}$ in the work of Gantert and  H\"{o}felsauer~\cite{GH18}. The first term on the right-hand side of inequality 5.9 in their paper is $\exp\big(-n\rho\min\left\{1-{x}/{x^*},1\right\}+o(n)\big)$.
They assumed that 
\[\rho\min\Big\{1-\frac{x}{x^*},1\Big\} \geq\inf\limits_{0<t\leq1}\Big\{\rho t+tI\Big(\frac{x-(1-t)x^*}{t}\Big)\Big\},\]
which does not hold for negatively large enough $x$. The proof of \hyperref[proofivv]{(iv)} and \hyperref[proofivv]{(v)} of  Theorem~\ref{ratefn} is essentially a corrected version of their techniques.

For $\theta_0<\theta<\infty$, the upper large deviation for the laws of $\{R_n^*(\theta,\mu)/n\}_{n\geq1}$ agrees with that of $\{R_n/n\}_{n\geq1}$ up to $\phi'(\theta)$.


\section*{Acknowledgement}
    This work is part of the author's Ph.D. dissertation and the  
	author wishes to thank	Antar Bandyopadhyay for suggesting the problem and also for various
	discussions which he had with him as the Ph.D. supervisor. 
	The author would also like to thank  the Council of Scientific and Industrial Research, Government of India and the Indian Statistical Institute, Kolkata
	for supporting his doctoral research.
	The author  also thanks the anonymous referee, whose careful reading and detailed comments have helped to improve the paper.




\providecommand{\bysame}{\leavevmode\hbox to3em{\hrulefill}\thinspace}
\providecommand{\MR}{\relax\ifhmode\unskip\space\fi MR }
\providecommand{\MRhref}[2]{%
  \href{http://www.ams.org/mathscinet-getitem?mr=#1}{#2}
}
\providecommand{\href}[2]{#2}


\end{document}